\documentclass [12pt,letterpaper,reqno]{amsart}

\usepackage{amssymb}

\DeclareMathSymbol{\twoheadrightarrow} {\mathrel}{AMSa}{"10}

\def\Q{{\mathbb Q}}
\def\Z{{\mathbb Z}}
\def\C{{\mathbb C}}

\def\F{{\mathbb F}}

\def\Gal{\mathrm{Gal}}

\def\Perm{\mathrm{Perm}}

\def\End{\mathrm{End}}

\def\Aut{\mathrm{Aut}}
\def\Hom{\mathrm{Hom}}

    \def\RR{\mathfrak{R}}

\def\fchar{\mathrm{char}}

\def\dim{\mathrm{dim}}

\newtheorem{thm}{Theorem}[section]
\newtheorem{lem}[thm]{Lemma}
\newtheorem{cor}[thm]{Corollary}
\newtheorem{prop}[thm]{Proposition}
\theoremstyle{definition}

\newtheorem{ex}[thm]{Example}

\newtheorem{rem}[thm]{Remark}

\title[Non-isogenous  superelliptic jacobians]
{Non-isogenous   superelliptic jacobians II}
\author[Yuri G. Zarhin]{Yuri G. Zarhin}
\address{Department of Mathematics, Pennsylvania State University,
University Park, PA 16802, USA}
\email{zarhin\char`\@math.psu.edu}
\thanks{The author  was partially supported by Simons Foundation Collaboration grant   \# 585711.
This work was done during his stay in 2022 and 2023 at the Max-Planck Institut f\"ur Mathematik (Bonn, Germany), whose hospitality and support are gratefully acknowledged.}

\begin{document}
\begin{abstract}
Let $\ell$ be an odd prime and $K$  a field of characteristic different from $\ell$. Let $\bar{K}$ be an algebraic closure of $K$. Assume that $K$ contains a primitive $\ell$th root of unity.
Let $n \ne \ell$ be another odd prime.
Let $f(x)$ and $h(x)$ be degree $n$ polynomials with coefficients in $K$ and without repeated roots. 

Let us consider superelliptic curves
$C_{f,\ell}: y^{\ell}=f(x)$ and $C_{h,\ell}: y^{\ell}=h(x)$ of genus $(n-1)(\ell-1)/2$, and their jacobians $J^{(f,\ell)}$ and $J^{(h,\ell)}$, which are
$(n-1)(\ell-1)/2$-dimensional abelian varieties  over $\bar{K}$. 

Suppose that one of the polynomials is irreducible and the other reducible over $K$.
We prove that if $J^{(f,\ell)}$ and $J^{(h,\ell)}$ are  isogenous over $\bar{K}$ then  both endomorphism algebras  $\mathrm{End}^{0}(J^{(f,\ell)})$ and   $\mathrm{End}^{0}(J^{(h,\ell)})$
contain an invertible element of multiplicative order $n$.
\end{abstract}
\dedicatory{To the memory of Kolya Vavilov}


\maketitle
\section{Definitions, notations, statements}
This paper is a follow up of \cite{ZarhinMZ06,ZarhinAGP} and we use their (more or less standard) notation.
 (See also \cite{ZarhinSh03,ZarhinMRL22,ZarhinAGP}.)
In particular,  $\ell$ is an odd prime, $\F_{\ell}$  the corresponding (finite) prime field of characteristic $\ell$. We write $\Z_{\ell}$ and $\Q_{\ell}$
for the ring of $\ell$-adic integers and the field $\Q_{\ell}$ of $\ell$-adic numbers respectively.
Let us fix a primitive $\ell$th root of unity
$$\zeta_{\ell} \in \C$$
in the field $\C$ of complex numbers.  We write $\Q(\zeta_{\ell})$ for the $\ell$th cyclotomic field and 
$$\Z[\zeta_{\ell}]=\sum_{i=0}^{\phi(q)-1}\Z\cdot \zeta_{\ell}^i$$
for its ring of integers. We write $\mathcal{P}_{\ell}(t)$ for the $\ell$th cyclotomic polynomial
$$\mathcal{P}_{\ell}(t)=\sum_{j=0}^{\ell-1}t^j\in \Z[t].$$

Let $K$ be  a field with $\mathrm{char}(K)\ne \ell$. Let us fix an  algebraic closure  $\bar{K}$ of $K$ and write $\Gal(K)=\Aut(\bar{K}/K)$ for the group of its $K$-linear automorpisms. In what follows we always assume that $K$ contains a primitive $\ell$th root of unity, say, $\zeta$.
Let $K_s\subset \bar{K}$ be the separable algebraic closure of $K$. The subfield $K_s$ of $\bar{K}$ is the natural group homomorphism (restriction map)
$$\Gal(K)=\Aut(\bar{K}/K) \to \Aut(K_s/K) =\Gal(K_s/K)$$
is a group isomorphism.

Let $X$  and $Y$ be  abelian varieties over $K$. By a theorem of Chow 
\cite[Th. 3.19]{Conrad}),
all $\bar{K}$-endomorphisms between $X$ and $Y$ are defined over $K_s$. In particular, all endomorphisms of $X$ are defined over $K_s$.
A more precise information about the field of definition of all endomorphisms of $X$ is given by a theorem of A. Silverberg 
\cite{Silverberg} (see Remark \ref{silver} below). 

Let $n \ge 3$ be an odd integer. Throughout the paper, we assume that $\ell$ does {\sl not} divide $n$.

Let 
$f(x)\in K[x]$ be a polynomial with coefficients
in  $K$, of  degree $n$ and  without repeated roots.
We write $\RR_f\subset \bar{K}$ for the $n$-element set of roots of $f(x)$,  $K(\RR_f)$ for the splitting field of $f(x)$ and
$\Gal(f/K)$ for the Galois group
$$\Gal(K(\RR_f)/K)=\Aut(K(\RR_f)/K)$$
of $f(x)$. As usual, one may view $\Gal(f/K)$ as a certain permutation subgroup of the group $\Perm(\RR_f)$ of all permutations of $\RR_f$.

We write $C_{f,\ell}$ for the smooth projective model of the plane affine curve $y^{\ell}=f(x)$.
It is well known
\cite{Poonen,SPoonen}
that the genus $g(C_{f,\ell})$  of $C_{f,\ell}$ is $(\ell-1)(n-1)/2$.
The map
 $$(x,y) \mapsto (x, \zeta y)$$
gives rise to a non-trivial biregular automorphism
$$\delta_{\ell}: C_{f,\ell} \to C_{f,\ell}$$
of period $\ell$ that is defined over $K$.

Let $J^{(f,\ell)}$  be the jacobian of $C_{f,\ell}$;
it is a $(\ell-1)(n-1)/2$-dimensional abelian variety that is defined over $K$.
We write $\End(J^{(f,\ell)})$ for the ring of $\bar{K}$-endomorphisms of 
$J^{(f,\ell)}$ and
$\End^0((J^{(f,\ell)})=\End(J^{(f,\ell)})\otimes\Q$ for the corresponding endomorphism algebra of $J^{(f,\ell)}$.
By functoriality,
$\delta_{\ell}$ induces a $K$-automorphism of $J^{(f,\ell)}$,
which we still denote by $\delta_{\ell}$.  It is known (\cite[p. 149]{Poonen}, \cite[p. 448]{SPoonen}; see also 
\cite{ZarhinGanita}) that
\begin{equation}
\label{deltaEquation}
 \mathcal{P}_{\ell}(\delta_{\ell})=\sum_{j=0}^{\ell-1} \delta_{\ell}^j=0
\end{equation}
in $\End(J^{(f,\ell)})$. 
Then \eqref{deltaEquation} gives rise to the  ring homomorphism,
\begin{equation}
\label{embedQ}
{\bf i}_{\ell,f}: \Z[\zeta_{\ell}] \hookrightarrow \Z[\delta_{\ell}] \subset \End(J^{(f,\ell)}), \ \zeta_{\ell}  \mapsto \delta_{\ell},
\end{equation}
which is a {\sl ring embedding} (\cite[p. 149]{Poonen}, \cite[p. 448]{SPoonen}; see also 
\cite{ZarhinGanita}). (Here $1 \in  \Z[\zeta_{\ell}]$ goes to the {\sl identity automorphism} $\mathrm{1}_{J^{(f,\ell)}}$ of $J^{(f,\ell)}$.)
This implies  that the subring $\Z[\delta_{\ell}]$ of $\End(J^{(f,\ell)})$ generated by $\delta_{\ell}$
is isomorphic to $\Z[\zeta_{\ell}]$. It follows that  the  $\Q$-subalgebra 
\begin{equation}
\label{Qdeltaq}
\Q[\delta_{\ell}] \subset\End^0(J^{(f,\ell)})
\end{equation}
 generated by $\delta_{\ell}$  
is canonically isomorphic to the $\ell$th cyclotomic field
$\Q(\zeta_{\ell})$
and therefore 
$$\dim_{\Q}(\Q[\delta_{\ell}] )=\ell-1.$$

Let $f(x)$ and $h(x)$ be degree $n$ polynomials with coefficients in $K$ and without repeated roots. Let
$$C_{f,\ell}: y^{\ell}=f(x), \ C_{h,\ell}: y^{\ell}=h(x)$$
be the corresponding genus $(n-1)(\ell-1)/2$ superelliptic curves over $K$, whose jacobians we denote by $J^{(f,\ell)}$ and $J^{(h,\ell)}$, respectively. These jacobians are $(n-1)(\ell-1)/2$-dimensional abelian varieties defined over $K$.
Assuming that the Galois  properties of roots of $f(x)$ and $h(x)$   are {\sl distinct} 
(see below), we will prove that if the abelian varieties  $J^{(f,\ell)}$ and $J^{(h,\ell)}$ are isogenous over $\bar{K}$
then they admit an ``additional symmetry'', i.e., the  inclusion \eqref{Qdeltaq} is {\sl not} an equality.

The main result of this paper is the following assertion.

\begin{thm}
\label{endoH}
Suppose that $n$  and $\ell$ are distinct odd primes.
Let $K$ be a field of characteristic different from
$\ell$. 
Let $f(x), h(x) \in K[x]$ be degree $n$ polynomials without repeated roots.  Suppose that one of the
polynomials is irreducible and the other is reducible.

If the corresponding superelliptic jacobians $J^{(f,\ell)}$ and $J^{(h,\ell)}$ are  isogenous over $\bar{K}$
then  both endomorphism algebras $\End^0(J^{(f,\ell)})$ and  $\End^0(J^{(h,\ell)})$
contain an invertible element of multiplicative order $n$.
In addition,
$$\dim_{\Q}\left(\End^0(J^{(f,\ell)})\right)=\dim_{\Q}\left(\End^0(J^{(h,\ell)})\right) \ge $$
$$(\ell-1)(n-1)=2\ \dim(J^{(f,\ell)})=2\ \dim(J^{(h,\ell)}).$$
\end{thm}

The next assertion may be viewed as a  partial generalization of  Theorem \ref{endoH} to the case of an arbitrary odd $n\ge 3$.

\begin{thm}
\label{endoH2}
Suppose that $n\ge 3$ is an odd integer and $\ell$ is an odd prime not dividing $n$.
Let $K$ be a field of characteristic different from
$\ell$. 
Let $f(x), h(x) \in K[x]$ be degree $n$ polynomials without repeated roots.  Suppose that $f(x)$ is irreducible over $K$.

Assume additionally that 
the order of the Galois group $\Gal(h/K)$ of $h(x)$ is prime to $n$.
(E.g., each irreducible factor of $h(x)$  over $K$ has degree $1$ or $2$.)

If the corresponding superelliptic jacobians $J^{(f,\ell)}$ and $J^{(h,\ell)}$ are  isogenous over $\bar{K}$ then 
there is a prime divisor $r$ of $n$ such that both endomorphism algebras $\End^0(J^{(f,\ell)})$ and  $\End^0(J^{(h,\ell)})$
contain an invertible element of multiplicative order $r$.
In addition,
$$\dim_{\Q}\left(\End^0(J^{(f,\ell)})\right)=\dim_{\Q}\left(\End^0(J^{(h,\ell)})\right) \ge
 (\ell-1)(r-1).$$
\end{thm}

\begin{rem}
If the conditions of  Theorem \ref{endoH2} hold then $h(x)$ is reducible over $K$, see \cite[Remark 1.4]{ZarhinAGP}.
\end{rem}

 \begin{cor}
  \label{doublePrime}
  Suppose that $n$  and $\ell$ are distinct odd primes.
  Let $f(x), h(x) \in K[x]$ be degree $n$ polynomials without repeated roots.  Suppose that $f(x)$ is irreducible over $K$ and $\Gal(\RR_f)$ is a doubly reansitive permutation group of $\RR_f$.
Assume also that $\Gal(h/K)$ is a cyclic group of order $n$.
If the corresponding superelliptic jacobians $J^{(f,\ell)}$ and $J^{(h,\ell)}$ are  isogenous over $\bar{K}$
then both endomorphism algebras $\End^0(J^{(f,\ell)})$ and  $\End^0(J^{(h,\ell)})$
contain an invertible element of multiplicative order $n$.
In addition,
$$\dim_{\Q}\left(\End^0(J^{(f,\ell)})\right)=\dim_{\Q}\left(\End^0(J^{(h,\ell)})\right) \ge$$
$$(\ell-1)(n-1)=2\ \dim(J^{(f,\ell)})=2\ \dim(J^{(h,\ell)}).$$
 \end{cor}

\begin{ex}
 Let $\ell$ be an odd prime and $n \ge 3$  an odd integer that is {\sl not} divisible by $\ell$. Let us take as $K$ the $\ell$th cyclotomic field
 $\Q(\zeta_{\ell})$. 
 Let us put
 $$f(x)=x^n-2, \ h(x)=x^n-1 \in  K[x].$$
 By the $2$-dic Eisenstein criterion, $f(x)$ is irreducible over $K$ (recall that prime $2$ is unramified in $\Q(\zeta_{\ell})=K$, because $\ell$ is odd) while $h(x)$ is obviously reducible over $K$.
 Let  $\ell$ be an odd prime that does {\sl not} divide $n$.
 The curves $C_{f,\ell}$ and $C_{h,\ell}$ are obviously isomorphic over $\bar{K}$. They both admit periodic $\bar{K}$-automorphisms $\tilde{\delta}_n$  of order $n$ defined by the (same) formula
 $$(x,y) \mapsto (\zeta_n x, y)$$
 where
 $$\zeta_n \in \bar{K}=\bar{\Q}\subset \C$$
 is a primitive $n$th root of unity.
 This implies that their jacobians $J^{(f,\ell)}$ and $J^{(h,\ell)}$ are abelian varieties over $\bar{K}$ that are isomorphic over $\bar{K}$ and 
 admit periodic automorphisms of  order $n$ that we continue to denote by $\tilde{\delta}_n$. It follows that  if $r$ is any prime divisor of $n$ then both
 $J^{(f,\ell)}$ and $J^{(h,\ell)}$ admit automorphisms $\left(\tilde{\delta}_n\right)^{n/r}$ of multiplicative order $r$ that may be viewed as invertible elements of multiplicative order $r$ in $\End^0(J^{(f,\ell)})$ and $\End^0(J^{(h,\ell)})$.
 \end{ex}
 
 \begin{ex}
 \label{SnAn}
  Let $n \ge 5$ be an odd integer  and $\ell$ an odd prime  not dividing $n$. Let us put
  $K=\Q(\zeta_{\ell})$.

  Let $f(x)\in K[x]$ be a degree $n$ irreducible polynomial over $K$, whose Galois group $\Gal(f/K)$ is either the full symmetric group $\mathrm{S}_n$ or the alternating group $\mathrm{A}_n$. It is known
  \cite{ZarhinCrelle,ZarhinCambridge,ZarhinLuminy} that the endomorphism algebra  $\End^0(J^{(f,\ell)})$ is isomorphic to the field $\Q(\zeta_{\ell})$; in particular, $J^{(f,\ell)}$ is {\sl absolutely simple}. The (cyclic) multiplicative group of all roots of unity in $\Q(\zeta_{\ell})$ has order $2\ell$, which is prime to (odd) $n$. Hence, 
  $\End^0(J^{(f,\ell)})\cong \Q(\zeta_{\ell})$ does not contain elements of multiplicative order $r$
  for any prime divisor $r$ of $n$.
  
  Let $h(x) \in K[x]$ be a degree $n$ polynomial that splits over $K$ into a product of linear factors.
  
  Then it follows from Theorem \ref{endoH2}  that $J^{(f,\ell)}$ and 
  $J^{(h,\ell)}$ are {\sl not} isogenous over $\bar{K}=\bar{\Q}$ (and therefore even over $\bar{\C}$,
  in light of a theorem of Chow \cite[Th. 3.19]{Conrad}). Since $J^{(f,\ell)}$ is absolutely simple and has the same dimension as $J^{(h,\ell)}$, it follows that every $\C$-homomorphism between $J^{(f,\ell)}$ and $J^{(h,\ell)}$ is zero.
 \end{ex}

\begin{ex}
\label{xNx1}
 Let $n \ge 5$ be an odd integer  and $\ell$ an odd prime not dividing $n$.
Let us consider the degree $n$ polynomials 
$$f_1(x)=x^n-x-1, \quad f_2(x)=\sum_{j=0}^n \frac{x^j}{j!}$$
with rational coefficients. It is known (Selmer, Osada \cite{Osada}) that
$f_1(x)$ is irreducible over $\Q$ and $\Gal(f_1/\Q)=\mathrm{S}_n$. By a theorem of Schur \cite{Coleman}, $f_2(x)$ is irreducible over $\Q$ and $\Gal(f_2/\Q)=\mathrm{S}_n$ or $\mathrm{A}_n$. 

Recall that $n \ge 5$ and therefore $\mathrm{A}_n$ is a simple non-abelian group that coincides with the commutator subgroup of $\mathrm{S}_n$; in addition, $\mathrm{A}_n$ is a maximal subgroup of $\mathrm{S}_n$. Since $K:=\Q(\zeta_{\ell})$ is an abelian extension of $\Q$, the Galois group $\Gal(f_k/K)$  is either $\mathrm{S}_n$ or $\mathrm{A}_n$ (for $k=1,2$). In particular, both $f_1$ and $f_2$ remain irreducible over $K$. 

It follows from Example \ref{SnAn} that if $h(x) \in K[x]$ is a degree $n$ polynomial that splits over $K$ into a product of linear factors
 then every $\C$-homomorphism between $J^{(f_k,\ell)}$ and $J^{(h,\ell)}$ is zero (for both $k=1,2$).
 
\end{ex}

\begin{rem}
Examples \ref{SnAn} and \ref{xNx1} illustrate the title of this paper, whose aim is to provide a  criterion  for certain superelliptic jacobians {\sl not} to be isogenous.
\end{rem}

\begin{rem}
In light of Theorem of Chow cited above, the assertions of Theorem \ref{endoH},  \ref{endoH2}
(and  of Theorem \ref{isogEll} below),
 and Corollary \ref{doublePrime} remain true if one replaces $\bar{K}$ by $K_s$.
\end{rem}

The paper is organized as follows. In Section \ref{order2} we recall basic facts about Galois properties of points of order $\ell$ 
on superelliptic jacobians.
We also state Theorem \ref{isogEll} that is  a slightly stronger version of Theorem \ref{endoH2}.
Section \ref{mainproof} contains the proof of Theorem \ref{isogEll}.
We prove Theorem \ref{endoH}  in Section \ref{PendoH}.
Corollary \ref{doublePrime} is proven in Section \ref{news}.

{\bf Acknowledgments}. I am deeply grateful to the referee for helpful comments. 

\section{Points of order $\ell$ on superelliptic jacobians}
\label{order2}
Recall that
$K_s \subset \bar{K}$ 
is
the separable algebraic closure of $K$.
Let  $X$  be a positive-dimensional abelian variety   over $K$.
We write $\End(X)$ for the ring of all $\bar{K}$-endomorphisms of $X$ and
$\End^0(X):=\End(X)\otimes \Q$ for the corresponding {\sl endomorphism algebra} of $X$, which is a semisimple {\sl finite-dimensional} $\Q$-algebra.
We write $\mathrm{1}_X$ for the identity automorphism of $X$ that is may be viewed as the {\sl identity element} of the
 $\Q$-algebra $\End^0(X)$.

 If 
$d$ is a positive integer  then we write $X[d]$ for the kernel of
multiplication by $d$ in $X(\bar{K})$.  Recall (\cite[Sect. 6]{Mumford}, \cite[Sect. 8, Remark 8.4]{Milne}) that  if $d$ is {\sl not} divisible
by $\fchar(K)$ then  $X[d]$ is a $\Gal(K)$-submodule of $X(K_s)$;
in addition, $X[d]$  is isomorphic as a commutative group to $(\Z/d\Z)^{2\dim(X)}$.

Let  $K(X[d])$ be the {\sl field of definition} of all torsion points of order dividing $d$ on
$X$.  It is well known \cite[Remark 8.4]{Milne} that $K(X[d])$ lies in $K_s$ and  is a finite Galois extension of $K$.
Let us put
$$\tilde{G}_{d,X,K}:=\Gal(K(X[d])/K).$$
One may view $\tilde{G}_{d,X,K}$ as a certain subgroup of $\Aut_{\Z/d\Z}(X[d])$ and $X[d]$ as a faithful $\tilde{G}_{d,X,K}$-module.
In addition, the structure of the $\Gal(K)$-module on $X[d]$ is induced by the canonical (continuous) {\sl surjective} group homomorphism
$$\tilde{\rho}_{d,X}:\Gal(K)\twoheadrightarrow \Gal(K(X[d])/K)=\tilde{G}_{d,X,K}.$$

For example, if $d$ is a prime $\ell\ne \fchar(K)$ then
$X[\ell]$ is a $2\dim(X)$-dimensional vector space over the 
field $\F_{\ell}=\Z/\ell\Z$, and the inclusion $\tilde{G}_{\ell,X,K}
\subset \Aut_{\F_{\ell}}(X[\ell])$ defines a {\sl faithful} linear
representation of the group  $\tilde{G}_{\ell,X,K}$ in the vector
space $X[\ell]$ over $\F_{\ell}$.

\begin{rem}
\label{silver}
If $d \ge 3$ is an integer {\sl not} divisible by $\fchar(K)$ then
all the endomorphisms of $X$ are defined over $K(X[d])$, by a theorem of A. Silverberg 
\cite[Th. 2.4]{Silverberg}. (See also 
 \cite{GK,Remond,Pip}.)
\end{rem}

 The following assertion may be viewed as a variant of Theorem \ref{endoH2}  (when $Y=J^{(h,\ell)}$).

\begin{thm}
\label{isogEll}
Let $\ell$ be an odd  prime and
$K$  a field of characteristic $\ne \ell$. Let $n \ge 3$ be an odd positive integer that is not divisible by $\ell$, and $f(x)\in K[x]$ a degree  $n$ irreducible polynomial
without repeated roots.  Let us put $X=J^{(f,\ell)}$, which is a $(n-1)(\ell-1)/2$-dimensional abelian variety  over $K$.

Let $Y$ be an abelian variety  over $K$ such that the order of $\tilde{G}_{\ell,Y,K}$ is prime to $n$.
(E.g., $K(Y[\ell])=K$ or this order is a power of $\ell$.)

Suppose that $X$ and $Y$ are isogenous over $\bar{K}$.

Then there is an odd prime $r$ dividing $n$ such that both endomorphism algebras $\End^0(X)$ and $\End^0(Y)$ contain 
an invertible element of multiplicative order $r$.
In addition,
$$\dim_{\Q}\left(\End^0(J^{(f,\ell)})\right)=\dim_{\Q}\left(\End^0(Y)\right) \ge 
(\ell-1)(r-1).$$


\end{thm}

\begin{rem}
Recall that the order of $\tilde{G}_{\ell,Y,K}$ coincides with the degree 
$[K(Y[\ell]):K]$ of the Galois extension $K(Y[\ell])/K$.
\end{rem}

We will prove  Theorem \ref{isogEll} in Section \ref{mainproof}. Our proof is based on the Galois properties of certain points of order $\ell$ on superelliptic jacobians $J^{(f,\ell)}$
that will be discussed in the next subsection.

\subsection{Galois properties}
\label{QRR}
In  this subsection we recall  an  explicit description of a certain important Galois submodule of  $J^{(f,\ell)}[\ell]$  \cite{Poonen,SPoonen}
for arbitrary separable $f(x)$,   assuming as usual that $\ell$ does {\sl not} divide $n$.

Let us start with the $n$-dimensional $\F_{\ell}$-vector space
$$\F_{\ell}^{\RR_f}=\{\phi:\RR_f \to \F_{\ell}\}$$
of all $\F_{\ell}$-valued functions on $\RR_f$. The  action of $\Perm(\RR_f)$ on $\RR_f$ provides $\F_{\ell}^{\RR_f}$ with the structure of a faithful $\Perm(\RR_f)$-module, which splits into a direct sum
\begin{equation}
 \label{splitting}
\F_{\ell}^{\RR_f}=\F_{\ell}\cdot {\bf 1}_{\RR_f}\oplus Q_{\RR_f}
\end{equation}
of the one-dimensional subspace $\F_{\ell}\cdot {\bf 1}_{\RR_f}$ of constant functions  and the $(n-1)$-dimensional {\sl heart} \cite{Klemm,Mortimer}
$$Q_{\RR_f}:=\{\phi:\RR_f \to \F_{\ell}\mid \sum_{\alpha\in\RR_f}\phi(\alpha)=0\}$$
(here we use that $n$ is not divisible by $\ell$). Clearly, the  $\Perm(\RR_f)$-module $ Q_{\RR_f}$ is faithful. It remains faithful if we view it as a $\Gal(f/K)$-module.  

\begin{rem}
 \label{selfDual}
 There is a nondegenerate $\Perm(\RR_f)$-invariant $\F_{\ell}$-bilinear pairing
$$\Psi: \F_{\ell}^{\RR_f} \times \F_{\ell}^{\RR_f} \to \F_{\ell}, \  \phi,\psi \mapsto \sum_{\alpha\in \RR_f}\phi(\alpha)\psi(\alpha)$$
and the splitting \eqref{splitting} is an orthogonal direct sum. Clearly, the restriction of $\Psi$ to $\F_2\cdot {\bf 1}_{\RR_f}$ is nondegenerate and therefore
the restriction of $\Psi$ to $Q_{\RR_f}$ is nondegenerate as well. This implies that the  $\Gal(f/K)$-module $Q_{\RR_f}$ and its  {\sl dual}
$\Hom_{\F_{\ell}}(Q_{\RR_f},\F_{\ell})$ are  {\sl isomorphic}. 
\end{rem}

The field inclusion $K(\RR_f)\subset K_s$ induces the {\sl surjective} continuous group homomorphism
$$\Gal(K)=\Gal(K_s/K)\twoheadrightarrow \Gal(K(\RR_f)/K)=\Gal(f/K),$$
which gives rise to the natural structure of the $\Gal(K)$-module on $Q_{\RR_f}$ such that the image of $\Gal(K)$ in $\Aut_{\F_{\ell}}(Q_{\RR_f})$ coincides with
$$\Gal(f/K)\subset \Perm(\RR_f)\hookrightarrow \Aut_{\F_{\ell}}(Q_{\RR_f}).$$  
In order to explain why the structure of the Galois module $Q_{\RR_f}$ is important, let us consider the subgroup (actually, the Galois submodule)
$$J^{(f,\ell)}[1-\delta_{\ell}]=\{z \in J^{(f,\ell)}(\bar{K})\mid \delta_l(z)=z\}$$
of $J^{(f,\ell)}(\bar{K})$. B. Poonen and E. Schaefer \cite{Poonen,SPoonen} observed that the Galois module $J^{(f,\ell)}[1-\delta_{\ell}]$ is a Galois submodule of
$J^{(f,\ell)}[\ell]$ and is isomorphic to $Q_{\RR_f}$. In particular,
$K(\RR_f)$ coincides with the {\sl field of definition} of all points of 
$J^{(f,\ell)}[1-\delta_{\ell}]$.

We will need the following  elementary assertion about homomorphisms of Galois modules related to $Q_{\RR_f}$.

\begin{lem}
\label{invTran}
Suppose that $f(x)$ is irreducible over $K$ and a prime $\ell$ does not divide $n$. Then:

\begin{itemize}
 \item [(i)]
$Q_{\RR_f}$ does not contain nonzero Galois-invariants.
\item [(ii)]
Every Galois-invariant linear functional $Q_{\RR_f} \to \F_{\ell}$ is zero.
\item [(iii)]
Let $W$ be a $\F_{\ell}$-vector space provided with the trivial action of $\Gal(K)$.
Then every homomorphism of  Galois modules $$Q_{\RR_f} \to W$$ is zero.
\item [(iv)]
Let $V$ be a finite-dimensional $\F_{\ell}$-vector space provided with a linear action of $\Gal(K)$ in such a way that every simple (Jordan-H\"older) subquotient of $V$ is a trivial Galois module. Then every homomorphism of the Galois modules $Q_{\RR_f} \to V$ is zero.
\item[(v)]
Let $V$ be a finite-dimensional $\F_{\ell}$-vector space provided with a linear action of $\Gal(K)$ in such a way that every simple (Jordan-H\"older) subquotient of $V$ is a trivial Galois module. 
Let $M$ be a finite-dimensional $\F_{\ell}$-vector space provided with a linear action of $\Gal(K)$ in such a way that there is a filtration
$$M_0 =\{0\} \subset M_1 \subset \dots \subset M_d =M$$
of $M$ by Galois submodules $M_i$ such that every quotient $M_{i+1}/M_i$ is isomorphic to $Q_{\RR_f}$.
Then every homomorphism of the Galois modules $M \to V$ is zero.
\end{itemize}
\end{lem}

\begin{proof}
Recall that the irreducibility means that the Galois group acts transitively on $\RR_f$.
Let $\phi \in Q_{\RR_f}$ be a  Galois-invariant function on $\RR_f$. The transitivity implies that $\phi$ is   constant. This means that  there is $c \in \F_{\ell}$  such that
$\phi(\alpha)=c$ for all $\alpha\in \RR_f$ and therefore (since $\phi \in Q_{\RR_f}$)
$$0=\sum_{\alpha\in \RR_f}\phi(\alpha)=n \cdot c,$$ 
i.e., $c=0$.
 This means
that $\phi \equiv 0$, which proves (i).
In order to prove the second assertion of Lemma, recall (Remark \ref{selfDual}) that  the Galois modules $Q_{\RR_f}$ and $\Hom_{\F_{\ell}}(Q_{\RR_f},\F_{\ell})$ are isomorphic. 
Now the second assertion of our Lemma follows from the already proven first one.
On the other hand, the third assertion is an immmediate corollary of the second one: one has only to choose a basis of $W$. 
 In order to prove (iv), we will use induction by $\dim(V)$. We may assume that the subspace $V \ne \{0\}$. It follows from our assumptions on the Galois module $V$ that the subspace $V_0=V^{\Gal(K)}$ of Galois invariants is not $\{0\}$ as well. 
 If $u:  Q_{\RR_f} \to V$ is a homomorphism of Galois modules then the induction assumption applied to the quotient $V/V_0$ implies that
 the induced homomorphism of Galois modules
 $$Q_{\RR_f} \to V/V_0, \ \phi \mapsto u(\phi)+V_0$$
 is zero. This means that $u(Q_{\RR_f})\subset V$.
 Now the desired result follows from (iii) applied to
 $$Q_{\RR_f} \to V_0, \ \phi \mapsto u(\phi) \in V_0,$$
 because $\Gal(K)$ acts trivially on $V_0$.
 
 In order to prove (v), let us use induction by $d$. Let 
 $u: M \to V$ be a homomorphism of Galois modules. Since $M_1$ is isomorphic to $Q_{\RR_f}$,
 it follows from (iv) that $u(M_1)=\{0\}$, i.e., there is a  
 homomorphism of Galois modules $u_1: M/M_1 \to V$  such that $u$ is the composition of
 $$M \to M/M_1, \ m \mapsto m+M_1$$
 and $u_1$. If $d=1$ then we are done. If $d>1$ then the desired result follows from the induction assumption applied to the filtered Galois module
 $$M/M_1=(M/M_1)_{d-1}\supset \dots  \supset M_1/M_1=\{0\}=(M/M_0)_0.$$
\end{proof}

{Towse}

\section{Isogenous superelliptic jacobians}
\label{mainproof}

We will deduce Theorem \ref{isogEll} from the following auxiliary statements.

\begin{lem}[See Lemma 3.1  of \cite{ZarhinAGP}]
 \label{orbits}
 Let $G$ be a transitive permutation group of a finite nonempty set $\RR$, and $H$ a normal subgroup of $G$.
 Then the number of $H$-orbits in $\RR$ divides both $\#(\RR)$ and the index $(G:H)$. In particular, if 
 $\#(\RR)$ and $(G:H)$ are relatively prime then $H$ acts transitively on $\RR$.
 
 \end{lem}

\begin{lem}[See Lemma 3.2  of \cite{ZarhinAGP}]
 \label{remainIrr} Let $f(x)$ be a degree $n$ irreducible polynomial over a field $K$ and without repeated roots.
 Let $K_1/K$ be a finite Galois field extension, whose degree is prime to $n$.
 Then $f(x)$ remains irreducible over $K_1$. In particular, the order of Galois group $\Gal(f/K_1)$ is divisible by $n$.
 \end{lem}
 
 \begin{lem}
 \label{nilpotent}
 Let $\ell$ be a prime, $F$ a field of  characteristic $\ell$, and $V$ a finite-dimensional vector space over $F$. Let $\mathcal{N}\subset \End_F(V)$ be a 
 linear nilpotent Lie subalgebra of $\End_F(V)$, and
 $$V^{\mathcal{N}}:=\{v \in V \mid x(v)=0 \ \forall x \in \mathcal{N}\}.$$
  Let $\sigma$ be a linear automorphism of finite order in $V$ that commutes with all linear operators from $\mathcal{N}$ and such  that the subspace $V^{\sigma}$ of all $\sigma$-invariants in $V$ contains $V^{\mathcal{N}}$.
 Then the order of $\sigma$ is either $1$ or a power of $\ell$.

 \end{lem}
 
 \begin{rem} We will apply the easy ``commutative'' case of  Lemma \ref{nilpotent} to $F=\F_{\ell}$,  $V=X[\ell]$, 
 and $$\sigma \in \Gal(K(X[\ell]/K(X[1-\delta_{\ell}]))\subset 
 \Aut_{\F_{\ell}}(X[\ell])$$
 where $X=J^{(f,\ell)}$.
 \end{rem}

 \begin{proof}[Proof of Lemma \ref{nilpotent}]
 Replacing $\sigma$ by its suitable power, we may assume that $\sigma$ is a  periodic automorphism, whose order is prime to $\ell$. We need to  prove that $\sigma$ is the identity map.
 
 Since $\sigma$ is obviously a semisimple linear operator, $V$ splits into a direct sum of $\sigma$-invariant subspaces
 $$V=V^{\sigma} \oplus W \ \text{ where } W=(1-\sigma)V.$$
 In particular,
 $$V^{\sigma} \cap W =\{0\}.$$
 Assume that $W \ne \{0\}$. We need to arrive to a contradiction.
 Since $\mathcal{N}$ commutes with $\sigma$, the subspace $W=(1-\sigma)V$ is $\mathcal{N}$-invariant. By Engel's theorem, there is a {\sl nonzero} vector $w \in W$ that is killed all linear operators from $\mathcal{N}$, i.e., $w \in V^{\mathcal{N}}$.
 This implies that $w \in W \in V^{\sigma}$. Hence,
 $$w \in V^{\sigma} \cap W =\{0\}.$$
 This implies that $w=0$, which gives us a desired contradiction.

 \end{proof}

 \begin{cor}
 \label{ellGroup}
 Let $X=J^{(f,\ell)}$. Then:
 \begin{itemize}
  \item [(i)]

 $K(X[\ell])/K(\RR_f)$ is a finite Galois $\ell$-extension, i.e., either $K(X[\ell])=K(\RR_f)$ or the Galois group $\Gal(K(X[\ell])/K(\RR_f))$ is a finite $\ell$-group. In particular, if the order of $\Gal(f/K)$ is prime to $n$ then the order of $\tilde{G}_{\ell,X,K}$ is also prime to $n$.
 \item [(ii)]
 The Galois module $X[\ell]$ admits a filtration
 $$M_{0}=\{0\}\subset M_{1}\subset \dots \subset M_{\ell-1}=X[\ell],$$
 such that each consecutive quotient $M_{i+1}/M_i$ is isomorphic to the Galois module $Q_{\RR_f}$.
 \end{itemize}
 \end{cor}
 
 \begin{proof}
 Let us put 
 $$X[1-\delta_{\ell}]:=J^{(f,\ell)}[1-\delta_{\ell}]\subset 
 J^{(f,\ell)}[\ell]=X[\ell].$$
 The ring $\Z[\delta_{\ell}]\otimes \Z_{\ell}=:\Z_{\ell}[\delta_{\ell}]$ acts naturally on the
$\ell$-adic Tate module $\mathrm{T}_{\ell}(X)$ of $X$. It is known \cite{Ribet}
that $\mathrm{T}_{\ell}(X)$ is a free $\Z_{\ell}[\zeta_{\ell}]$-module of  rank
$$\frac{2\dim(J^{(f,\ell)})}{[\Q(\zeta_{\ell}):\Q]}=\frac{(n-1)(\ell-1)}{(\ell-1)}=
n-1.$$ 
In particular, the natural ring homomorphism
$$\Psi_{\ell}:\Z[\delta_{\ell}]/\ell=\Z_{\ell}[\delta_{\ell}]/\ell \to \End_{\F_{\ell}}(X[\ell])$$
is a ring {\sl embedding} that makes $X[\ell]$ a free $\Z[\delta_{\ell}]/\ell$-module of rank $n-1$. This implies that
$$X[1-\delta_{\ell}]=(1-\delta_{\ell})^{\ell-2}X[\ell],$$
because in the cyclotimic ring $\Z[\zeta_{\ell}]$ we have the equalities of ideals
$$\ell\Z[\zeta_{\ell}]=(1-\zeta_{\ell})^{\ell-1}\Z[\zeta_{\ell}], \quad
(1-\zeta_{\ell})^{\ell-2}\Z[\zeta_{\ell}]=\{z \in \Z[\zeta_{\ell}]\mid (1-\zeta_{\ell})z \in \ell \Z[\zeta_{\ell}]\}.$$
Now first assertion of (i) follows from Lemma \ref{nilpotent} applied 
to $$F=\F_{\ell},  \ V=X[\ell], \
 \mathcal{N}=\Psi_{\ell}((1-\delta_{\ell})\Z[\delta_{\ell}]/\ell),$$
 and $$\sigma \in \Gal(K(X[\ell]/K(X[1-\delta_{\ell}))\subset \Aut_{\F_{\ell}}(X[\ell]).$$
 The second one follows readily from the equality
$$[K(Y[\ell]):K]=[K(Y[\ell]):K(\RR_h)] \cdot [K(\RR_h):K] \ \text{)}$$
(recall that the prime $\ell$ does {\sl not} divide $n$).

 In order to prove (ii), let us put
 $$M_i:=(1-\delta_{\ell})^{\ell-1-i}X[\ell]\subset X[\ell].$$
 The freeness of the $\Z[\delta_{\ell}]/\ell$-module $X[\ell]$ implies that
 $M_i$ coincides with the kernel of
 $$(1-\delta_{\ell})^i: X[\ell] \to X[\ell].$$
 In particular,
 $$M_0=\{0\}, \ M_1=X[1-\delta_{\ell}] \cong Q_{\RR_f}, \  M_{\ell-1}=X[\ell].$$
 It is also clear that $(1-\delta_{\ell})^i$ induces an isomorphism of Galois modules
 $M_{i+1}/M_i$ and $M_1/M_0 \cong Q_{\RR_f}$. This ends the proof of (ii).
 
 \end{proof}
 \begin{lem}
 \label{isogenyNotInv}
 We keep the notation and assumptions of Theorem \ref{isogEll}.
 
 Suppose that  $K=K(Y[\ell])$.  Then there is a nontrivial group homomorphism
 $$\chi: \Gal(K(X[\ell])/K) \to \End^0(Y)^{*},$$
 whose image
 $$\Gamma:=\mathrm{Im}(\chi)\subset \End^0(Y)^{*}$$
 is a finite group that enjoys the following property.  
 
 The integers $n$ and  $\#(\Gamma)$ are not relatively prime.
 In other words,
 there is a prime $r \ne \ell$ that divides both $n$ and  $\#(\Gamma)$.
 In particular,   both endomorphism algebras $\End^0(J^{(f,\ell)})$ and  $\End^0(Y)$
contain an invertible element of multiplicative order $r$.
  \end{lem}
  
  \begin{lem}
  \label{Qln}
  Let $\ell$ and $r$ be two distinct odd primes.
  Let $\mathcal{V}$ be a nonzero finite-dimensional vector space over $\Q$. Let $A$ and $B$ be two commuting automorphisms of $\mathcal{V}$ that enjoy the following properties.
  \begin{itemize}
  \item[(i)]
  $A^{\ell}=B^r=\mathrm{1}_{\mathcal{V}}$ where $\mathrm{1}_{\mathcal{V}}$ is the identity automorphism of $\mathcal{V}$.
  \item[(ii)]
  $A-\mathrm{1}_{\mathcal{V}}$ is an automorphism of $\mathcal{V}$.
  
   \item[(iii)]
   $B \ne \mathrm{1}_{\mathcal{V}}$.
  \end{itemize}
  Then $\dim_{\Q}(\mathcal{V}) \ge (\ell-1)(r-1)$.
  
  \end{lem}

 \begin{proof}[Proof of Theorem \ref{isogEll} (modulo  Lemmas \ref{isogenyNotInv}
 and \ref{Qln})]
 Recall that $X=J^{(f,\ell)})$.
 
 It follows from Lemma \ref{remainIrr} that $f(x)$ remains irreducible over $K(Y[\ell])$.
 So, replacing $K$ by $K(Y[\ell])$, we may and will assume that $K(Y[\ell])=K$.
 Now it follows  from  Lemma \ref{isogenyNotInv}) that there is a prime $r\ne \ell$ that divides  $n$ and enjoys the following property.
 
 Both endomorphism algebras $\End^0(X)=\End^0(J^{(f,\ell)})$ and  $\End^0(Y)$
contain an invertible element of multiplicative order $r$.
In order to finish the proof of our Theorem, we need to prove the inequality
$$\dim_{\Q}(\End^0(X)) \ge (\ell-1)(r-1).$$
Let us use Lemma \ref{Qln} applied to $\mathcal{V}=\End^0(X)$. We define the automorphisms $A, \ B: \mathcal{V} \to \mathcal{V}$ of the $\Q$-vector space $\mathcal{V}=\End^0(X)$ as 
$$v \mapsto \delta_{\ell} v \ \text{ and } v \mapsto v u \quad \forall v \in \mathcal{V}=\End^0(X)$$
respectively. Clearly, $A$ and $B$ are commuting automorphisms of  $\mathcal{V}$ such that 
both $A^{\ell}$ and $B^r$ coincide with the {\sl identity automorphism} $\mathrm{1}_{\mathcal{V}}$ of 
$\mathcal{V}$. Since $u$ has multiplicative order $r>1$, $B \ne \mathrm{1}_{\mathcal{V}}$. On the
other hand, we know that the $\Q$-subalgebra 
$\Q[\delta_{\ell}]$ of  $\End^0(J^{(f,\ell)})=\End^0(X)$
is a subfield with the same identity element as $\End^0(X)$.
This implies that the nonzero $\delta_{\ell}-\mathrm{1}_X$ is an {\sl invertible} element of the $\Q$-algebra $\End^0(X)$. It follows that $A-\mathrm{1}_{\mathcal{V}}$ is an invertible automorphism of the $\Q$-vector space $\mathcal{V}$. So, $A$ and $B$
satisfy all the conditions of  Lemma \ref{Qln}. Applying Lemma \ref{Qln}, we conclude that
$\dim_{\Q}(\End^0(X)) \ge (\ell-1)(r-1)$, which ends the proof
of  Theorem \ref{isogEll}.


 \end{proof}

\begin{proof}[Proof of Lemma \ref{isogenyNotInv}]
In light of the theorem of Silverberg (Remark \ref{silver}(ii)), all endomorphisms of $Y$ are defined over $K$.
Applying this theorem  (see Remark \ref{silver}(ii) above) to $X\times Y$, we conclude that all the homomorphisms from $X$ to $Y$ are defined over $K(X[\ell])$.

Let $\mu: X \to Y$ be an isogeny. Dividing, if necessary, $\mu$ by a suitable power of $\ell$, we may and will assume that
\begin{equation}
\label{notZero}
\mu(X[\ell]) \ne \{0\}.
\end{equation}
Let us put
$$G_{\ell}:=\tilde{G}_{\ell,X,K}=\Gal(K(X[\ell])/K), \ G=\Gal(K(X[1-\delta_{\ell}])/K)=\Gal(f/K).$$
We know that $\mu$ is defined over $K(X[\ell])$. This allows us to define for each $\sigma \in G_{\ell}$ the isogeny
$\sigma(\mu):X \to Y$, which is the Galois-conjugate of $\mu$ (recall that both $X$ and $Y$ are defined over $K$). 
Then the same construction as in \cite[Sect. 4, proof of Prop. 2.4]{ZarhinMRL22} allows us to define a map
$$c: G_{\ell} \to \End^0(Y)^{*}, \ \sigma \mapsto c(\sigma)$$
where $c(\sigma)$ is determined by
$$\sigma(\mu)=c(\sigma)\mu \ \forall \sigma \in G_{\ell}=\Gal(K(X[\ell])/K).$$
We have for each $\sigma,\tau \in G_{\ell}$
$$c(\sigma\tau)\mu=\sigma\tau(\mu)=\sigma(\tau(\mu))= \sigma(c(\tau)\mu)=c(\tau) \sigma(\mu)=c(\tau)c(\sigma)\mu$$
(here we use that all elements of $\End(Y)$ are defined over $K$, i.e., are $G_{\ell}$-invariant).
Therefore
$$c(\sigma\tau)=c(\tau)c(\sigma) \ \forall \sigma, \tau \in G_{\ell}=\Gal(K(X[\ell])/K).$$
This means that the map
$$\chi: G_{\ell}=\Gal(K(X[\ell])/K) \to \End^0(Y)^{*}, \ \sigma \mapsto \chi(\sigma)=c(\sigma)^{-1}$$
is a {\sl group homomorphism}. Let $\Gamma \subset \End^0(Y)^{*}$ be the image of $\chi$, which is a finite
subgroup of $\End^0(Y)^{*}$.  We need to check that there is a prime divisor $r$ of $n$ that divides $\#(\Gamma)$.

Let $H_{\ell}$ be the kernel of $\chi$, i.e.,
\begin{equation}
\label{H4}
H_{\ell}=\{\sigma \in G_{\ell}\mid \sigma(\mu)=\mu\}.
\end{equation}
By definition, $H_{\ell}$ is a normal subgroup of $G_{\ell}$. Let $H$ be the image of $H_{\ell}$ in $G$ under the  natural {\sl surjective} group homomorphism
$$G_{\ell}=\Gal(K(X[\ell])/K)\twoheadrightarrow \Gal(K(X[1-\delta_{\ell}])/K)=$$
$$\Gal(K(\RR_f)/K)=\Gal(f/K)=G$$
induced by the  inclusion 
$$K(\RR_f)=K(X[1-\delta_{\ell}])\subset K(X[\ell])$$
of Galois extensions of $K$.
The surjectiveness implies that  $H$ is a normal subgroup of $G$ and the index $(G:H)$ divides 
$$(G_{\ell}:H_{\ell})=\#(\Gamma).$$

In order to finish the proof, we  need the following assertion that will be proven at the end of this section.

\begin{prop}
\label{GneH}
The subgroup $H$ of $G$ is not transitive on $\RR_f$. 
\end{prop}

{\bf End of Proof of Lemma \ref{isogenyNotInv} (modulo Proposition \ref{GneH})}
Combining Proposition \ref{GneH} with Lemma \ref{orbits}, we conclude that $(G:H)$ is {\sl not} prime to $n$.
Hence, there is a prime $r$ that divides both $(G:H)$ and $n$. Since   $n$ is prime to $\ell$ and $(G:H)$ divides 
$(G_{\ell}:H_{\ell})$, we conclude that  $r \ne \ell$ and $r$ divides $(G_{\ell}:H_{\ell})=\#(\Gamma)$. This ends the proof.

\end{proof}

 \begin{proof}[Proof of  Proposition \ref{GneH}]
 Suppose that $H$ is   transitive. Then $f(x)$ remains {\sl irreducible} over the subfield $E:=K(\RR_f)^{H}$ of all $H$-invariants in $K(\RR_f):=F$.  Clearly, $E/K$ is a  Galois extension and
 $$K\subset E \subset F=K(\RR_f)=E(\RR_f).$$
 Let us consider the subfield $L:=K(X[\ell])^{H_{\ell}}$ of all $H_{\ell}$-invariants in  $K(X[\ell])$, which is also a Galois extension of $K$.
 We have
 $$K(X[\ell])=L(X[\ell]), \quad H_{\ell}=\Gal(K(X[\ell])/L)=\Gal(L(X[\ell])/L).$$ 
 In addition,
 $$K\subset L=K(X[\ell])^{H_{\ell}}\supset K(\RR_f)^H=F^{H}=E\supset K$$
 and
 $$E=F^{H} =F\cap K(X[\ell])^{H_{\ell}}=F\cap L,$$
 becase $H$ is the image of $H_{\ell}\subset \Gal(K(X[\ell])/K)$
 in $\Gal(K(\RR_f)/K)=\Gal(F/K)$.
 
 We want to prove that $f(x)$ remains {\sl irreducible} over $L$. In order to do it,
 notice that $H_{\ell}$ acts transitively on $\RR_f$. Indeed, let $\alpha,\beta$ be  elements of $\RR_f$.
 By our assumption, there is 
 $\sigma \in H\in \Gal(K(\RR_f)/K)$ such that $\sigma(\alpha)=\beta$. Pick a field automorphism
 $$\sigma_{\ell} \in H_{\ell} \subset \Gal(K(X[\ell])/L)\subset \Gal(K(X[\ell])/K)$$
 such that the restriction of $\sigma_{\ell}$ to $K(\RR_f)$ coincides with $\sigma$. Since
 $$\alpha \in \RR_f \subset K(\RR_f),$$
 we get
 $$\sigma_{\ell}(\alpha)=\sigma(\alpha)=\beta.$$
 This proves the transitivity of of the action of $H_{\ell}$ on the set $\RR_f$ of roots of $f(x)$.
  It follows that $f(x)$ is {\sl irreducible} over
 the field $K(X[\ell])^{H_{\ell}}=L$.
 
 
 
  Replacing $K$ by its overfield $L=K(X[\ell])^{H_{\ell}}$, we may and will assume that 
  $$H_{\ell}=\Gal(K(X[\ell])/K).$$
  In particular,  
  \begin{equation}
  \label{muGalois}
  \sigma(\mu)=\mu\  \forall \sigma \in H_{\ell}=\Gal(K(X[\ell])/K).
  \end{equation}
  Recall that
  \begin{equation}
  \label{sigmaX4}
  \sigma(\mu)(\sigma(x))=\sigma(\mu(x))\  \forall \sigma \in \Gal(K(X[\ell])/K), \
  x \in X(K[\ell]).
  \end{equation}
  Since $X[\ell] \subset  X(K[\ell])$ and {\sl nonzero} $\mu(X[\ell])$ obviously lies in $Y[\ell]$, we conclude that the map
  \begin{equation}
  \label{XellYell}
   X[\ell] \to Y[\ell] , \ x \mapsto \mu(x)
   \end{equation}
  is a {\sl nonzero} homomorphism of $\Gal(K)$-modules. Recall that we assume  that  the Galois action on $Y[\ell]$ is {\sl trivial}.
  On the other hand, in light of Corollary \ref{ellGroup}, the Galois module $X[\ell]$ admits a filtration, all whose consecutuve quotients
  are isomorphic to $Q_{\RR_f}$. Since $f(x)$ is irreducible over $K$, it follows from  Lemma \ref{invTran} that the homomorphism \eqref{XellYell} is {\sl zero},
  which is not the case. The obtained contradiction proves that $H$ is not transitive. 
\end{proof}

\begin{proof}[Proof of Lemma \ref{Qln}]
Clearly, $B$ is a semisimple (i.e., diagonalizable over $\bar{\Q}$) linear operator in $\mathcal{V}$. The same is obviously true for the linear operator $B-\mathrm{1}_{\mathcal{V}}:\mathcal{V} \to \mathcal{V}$. The semisimplicity of $B-\mathrm{1}_{\mathcal{V}}$ implies that $\mathcal{V}$ splits into a direct sum
$$\mathcal{V}=\left(B-\mathrm{1}_{\mathcal{V}}\right)(\mathcal{V})\oplus \ker\left(B-\mathrm{1}_{\mathcal{V}}\right)$$
of the image $\left(B-\mathrm{1}_{\mathcal{V}}\right)(\mathcal{V})$ and the kernel $\ker\left(B-\mathrm{1}_{\mathcal{V}}\right)$
of $B-\mathrm{1}_{\mathcal{V}}$;
clearly, these two subspaces are $B$-invariant.
 In addition, the restriction of $B-\mathrm{1}_{\mathcal{V}}$ to the subspace
$$\mathcal{V}_0:=\left(B-\mathrm{1}_{\mathcal{V}}\right)(\mathcal{V})$$
is an automorphism of $\mathcal{V}_0$. Recall that our conditions on $B$ imply that
$$\mathcal{V}_0 \ne \{0\}.$$

Since $A$ and $B$ commute, both subspaces $\left(B-\mathrm{1}_{\mathcal{V}}\right)(\mathcal{V})=\mathcal{V}_0$
and  $\ker\left(B-\mathrm{1}_{\mathcal{V}}\right)$.
are also $A$-invariant. Let
$$A_0, \ B_0: \mathcal{V}_0 \to \mathcal{V}_0$$
be the restrictions to $\mathcal{V}_0$ of $A$ and $B$ respectively. Clearly, $A_0$ and  
$B_0$ commute, and both $A_0^{\ell}$ and $B_0^r$ coincide with the {\sl identity automorphism} $\mathrm{1}_{\mathcal{V}_0}$ of $\mathcal{V}_0$. In addition, both
$A_0-\mathrm{1}_{\mathcal{V}_0}$ and $B-\mathrm{1}_{\mathcal{V}_0}$ are automorphisms of $\mathcal{V}_0$.

Let 
$$\mathcal{P}_{A_0}(t),  \ \mathcal{P}_{B_0}(t) \in \Q[t]$$
be the {\sl minimal polynomials} of $A_0:\mathcal{V}_0 \to \mathcal{V}_0$ and $B_0: \mathcal{V}_0 \to \mathcal{V}_0$ respectively. Both minimal polynomials are monic of positive degree, and all their coefficients are rational numbers. Clearly,
$\mathcal{P}_{A_0}(t)$ divides $t^{\ell}-1$ and  $\mathcal{P}_{B_0}(t)$ divides $t^r-1$; in addition, $t-1$ divides neither $\mathcal{P}_{A_0}(t)$ nor $\mathcal{P}_{B_0}(t)$. 
Recall that
$$t^{\ell}-1 =(t-1)\Phi_{\ell}(t), \quad t^r-1 =(t-1)\Phi_{\ell}(t)$$
where $\Phi_{\ell}(t)$ and $\Phi_r(t)$ are $\ell$th and $r$th cyclotomic polynomials respectively; they both are irreducible over $\Q$. It follows that
$$\mathcal{P}_{A_0}(t)= \Phi_{\ell}(t),  \quad \mathcal{P}_{B_0}(t) =\Phi_r(t).$$
This implies that the $\Q$-subalgebra $\Q[A_0]$ of $\End_{\Q}(\mathcal{V}_0)$ generated by $A_0$ is isomorphic to the $\ell$th cyclotomic field
$$\Q[t]/\Phi_{\ell}(t)\Q[t] \cong \Q(\zeta_{\ell}).$$
Similarly, the $\Q$-subalgebra $\Q[B_0]$ of $\End_{\Q}(\mathcal{V}_0)$ generated by $B_0$ is isomorphic to the $r$th cyclotomic field
$$\Q[t]/\Phi_r(t)\Q[t] \cong \Q(\zeta_r).$$
Since $A_0$ and $B_0$ commute, the (commutative) $\Q$-subalgebras  $\Q[A_0]$ and $\Q[B_0]$ also commute. This implies that the {\sl nonzero} $\Q$-vector space $\mathcal{V}_0$ carries the natural structure of a module over the $\Q$-algebra
$$\Q[A_0]\otimes_{\Q}\Q[B_0] \cong \Q(\zeta_{\ell})\otimes_{\Q}\Q(\zeta_r).$$
Since $r$ and $\ell$ are distinct odd primes, the cyclotomic fields
$\Q(\zeta_{\ell})$ and ${\Q}\Q(\zeta_r)$ are linearly disjoint over $\Q$; actually, this tensor product is canonically isomorphic to $\ell r$th cyclotomic field $\Q(\zeta_{\ell r})$ of degree
$(\ell-1)(r-1)$. It follows that $\mathcal{V}_0$ carries the natural structure of 
a  $\Q(\zeta_{\ell r})$-vector space. Hence, 
$\dim_{\Q}(\mathcal{V}_0)$ is divisible by the degree
$$[\Q(\zeta_{\ell r}):\Q]=(\ell-1)(r-1).$$
Since $\mathcal{V}_0 \ne \{0\}$, we conclude that
$\dim_{\Q}(\mathcal{V}_0)\ge (\ell-1)(r-1).$ Taking into account that $\mathcal{V}_0$ is a subspace of $\mathcal{V}$, we conclude that
$$\dim_{\Q}(\mathcal{V})\ge \dim_{\Q}(\mathcal{V}_0)\ge (\ell-1)(r-1).$$
This ends the proof of our Lemma.
\end{proof}

\section{Proof of  Theorem \ref{endoH}}
\label{PendoH}
So, $n$ is an odd {\sl prime}, both $f(x)$ and $h(x)\in K[x]$ are degree $n$ polynomials without repeated roots,
$f(x)$ is irreducible and $h(x)$ is reducible. Since $n$ is a prime,  the reducibility of $h(x)$ implies that 
the order of $\Gal(h/K)$ is prime to $n$ (see \cite[Lemma 2.6]{ZarhinMRL22}).
Let us put $Y=J^{(h,\ell)}$. We are given that the degree $[K(\RR_h):K]$ of the field extension $K(\RR_h)/K$ is {\sl not} divisible by $n$. Applying Corollary \ref{ellGroup} to  $Y$ and $h(x)$ (instead of $X=J^{(f,\ell)}$ and $f(x)$), we conclude that 
the order of the group  $\tilde{G}_{\ell,Y,K}$ 
is prime to $n$.
Now the desired result follows readily from Theorem \ref{isogEll}, because if $r$ is a prime divisor of $n$ then $r=n$.

  \section{Doubly transitive and cyclic Galois groups}
  \label{news}

  \begin{proof}[Proof of Corollary \ref{doublePrime}]
   By definition of the field 
   $$K_h:=K(\RR_h),$$
   the polynomial $h(x)$ splits into a product of linear factors over $K_h$. Recall that $\Gal(K_h/K)=\Gal(h/K)$ is a {\sl cyclic} group of prime order $n$ and $\Gal(f/K)=\Gal(K(\RR_f)/K)$ is {\sl doubly transitive}.
   In light of Proposition 1.8 of \cite{ZarhinMRL22}, the field extensions $K(\RR_f)/K$ and $K_h/K$ are {\sl linearly disjoint}.  This implies that $f(x)$ remains irreducible over $K_h$. Now the desired result follows readily from Theorem \ref{endoH}
   applied to $f(x)$ and $h(x)$ over $K_h$ (instead of $K$).
  \end{proof}

\end{document}